\documentclass[12pt]{article}
\usepackage{a4}
\usepackage{amsmath}
\usepackage{amssymb}
\usepackage{amsfonts}
\usepackage{amsthm}
\usepackage{graphicx}
  \newtheorem{thm}{Theorem}[section]
 
 \newtheorem{prop}[thm]{Proposition}

 \newtheorem{rem}[thm]{Remark}

 \DeclareMathOperator{\supp}{supp}
\newcommand{\C}{\mathbb{C}}
\newcommand{\R}{\mathbb{R}}

\def\C{\mathbb{C}}
\def\R{\mathbb{R}}

\def\D{\mathbb{D}}

\author {Christian Berg}
\title {Indeterminate Stieltjes moment problems revisited}

\date{11.7.2024}

\begin{document}
\maketitle

\begin{center}
{\bf Dedicated to Mourad Ismail at the occasion of his 80'th birthday}  
\end{center}

\medskip
\begin{abstract} We consider a normalized indeterminate Hamburger moment sequence $s$ which is supposed to be Stieltjes. We revisit old results about determinacy/indeterminacy in the sense of Stieltjes for $s$ and we prove some new results about the concepts involved.
\end{abstract}

{\bf Mathematics Subject Classification}: Primary 44A60; Secondary 30B50

{\bf Keywords}. Stieltjes and Hamburger moment problems. Indeterminacy. Friedrichs solution to a Stieltjes problem.

\section{Introduction}

In a recent paper  Sokal and Walrad \cite{S:W} gave a continued fraction characterization of Stieltjes moment sequences for which there exists a representing measure with support in $[\xi,\infty)$. In Remark 3 of their paper the authors referred to a private communication by the present author. The purpose of the present paper is to make this communication available.  

For a normalized Stieltjes moment sequence $s$ which is indeterminate as a Hamburger moment sequence, there is a  quantity $\alpha=\alpha(s)\le 0$, which  
can be used to characterize Stieltjes indeterminacy of $s$. In fact there is exactly one measure (resp. several measures) supported by $[0,\infty)$  with moment sequence $s$ if  $\alpha(s)=0$ (resp. $\alpha(s)<0$). Chihara \cite[Lemma 1]{Ch1} proved an equivalent of
\begin{equation}\label{eq:alpha1}
\alpha(s)=\lim_{n\to\infty}\frac{P_n(0)}{Q_n(0)}
\end{equation}
in terms of chain sequences,
where $P_n$ and $Q_n$ are the orthonormal polynomials and those of the second kind affiliated with $s$. In \cite{B:V} there was given another characterization of $\alpha(s)$ as
\begin{equation}\label{eq:alpha2}
\alpha(s)=\lim_{x\to -\infty}\frac{D(x)}{B(x)},
\end{equation}
where $B,D$ are two of the Nevanlinna functions for the indeterminate Hamburger problem,
see Section 2  for definitions in connection with the Nevanlinna parametrization \eqref{eq:Nev-par} of the solutions to the indeterminate Hamburger moment problem. The formula \eqref{eq:alpha2} for $\alpha(s)$ together with the graph of the function $j(x)=D(x)/B(x)$ in Section 3 makes it transparent that the zeros of $B(z)t-D(z)$ are bounded below, hence of the form $x_1(t)<x_2(t)<\ldots$ for a sequence $(x_n(t))$ tending to infinity. When $t=\infty$ then $B(z)t-D(z)$ shall be interpreted as $B(z)$. Therefore, the
 N-extremal solutions $\nu_t, t\in\R\cup\{\infty\}$ of the Hamburger moment problem are of the form
\begin{equation}\label{eq:Ne1}
\nu_t=\sum_{n=1}^\infty \rho(x_n(t))\varepsilon_{x_n(t)},
\end{equation}
where 
\begin{equation}\label{eq:Ne2}
\rho(x)=\left(\sum_{n=0}^\infty P_n^2(x)\right)^{-1},\quad x\in\R.
\end{equation}
By $\varepsilon_a$ we denote the measure with mass 1 concentrated in the point $a\in\R$.

For any solution $\mu$ of the Hamburger moment problem we define
\begin{equation}\label{eq:Ne3}
\xi(\mu):=\inf \supp(\mu).
\end{equation}
We therefore get
$$
\supp(\nu_t)=\{x_n(t), n\ge 1\},\quad \xi(\nu_t)=x_1(t).
$$
In particular, the graph of $j(x)$  shows that $\supp(\nu_t)$ is contained in $[0,\infty)$ if and only if $\alpha(s)\le t\le 0$. The last  result can also be deduced from \cite[p. 340]{Ch2}. Notice that there are solutions $\mu$ with $\supp(\mu)=\R$ and hence $\xi(\mu)=-\infty$. This is true for the solution 
$$
\mu=\frac{1}{\pi}\left(B^2(x)+D^2(x)\right)^{-1}dx,
$$
cf. formula (2.15) in \cite{B:V}.

The standard orthonormal basis of $\ell^2$ is denoted $e_0, e_1,\ldots$. 
We consider the Jacobi matrix \eqref{eq:Jac} as an operator in $\ell^2$ with the domain
$\mathcal F$ of sequences with only finitely many non-zero elements. 
It is a symmetric operator of deficiency indices $(1,1)$. Its self-adjoint extensions in $\ell^2$ can be parametrized as $(T_t, D(T_t)), t\in\R\cup\{\infty\}$ such that
\begin{equation}\label{eq:Ne4}
\nu_t=\langle E_t e_0, e_0\rangle, t\in\R\cup\{\infty\},
\end{equation}
where $E_t$ is the spectral measure of $T_t$. The operators $T_t$ are bounded below with the lower bound
\begin{equation}\label{eq:Ne5}
\xi(\nu_t)=x_1(t)=\inf\{\langle T_t e,e\rangle \mid e\in D(T_t), ||e||=1\}.
\end{equation}
 
Later the two N-extremal solutions
 $\nu_0$ and $\nu_{\alpha(s)}$ were identified with respectively the Krein and the Friedrichs self-adjoint extensions of the Jacobi matrix. As far as we know the measure $\nu_{\alpha(s)}$ was first associated with the  Friedrichs extension
in Pedersen's paper \cite{Pe1}.

For information  about the Friedrichs extension in general see \cite{GKP}. It has the same lower bound as the Jacobi matrix, viz.
$$
x_1(\alpha(s))=\inf\{\langle Je, e\rangle \mid e\in\mathcal F, ||e||=1\}.
$$
 
 \begin{rem}{\rm The self-adjoint operators $(T_t, D(T_t))$ are parametrized differently in \cite[Theorem 6.23]{Sch} because of a different form of the Nevanlinna parametrization. See Remark 7.14 in \cite{Sch}.
 }
 \end{rem}
 
 We prove that $\alpha(s)$ can also be calculated as
 $$
 \alpha(s)=\lim_{x\to -\infty}\frac{C(x)}{A(x)},
 $$
 where $A,C$ are the two other Nevanlinna functions.
 
 Let $x_{n,k}, k=1,\dots,n$ be the zeros of $P_n$ in increasing order. By the interlacing of the zeros of $P_n$ and $P_{n+1}$ we have that $(x_{n,k})_{n\ge k}$ is decreasing in $n$ for each $k\ge 1$, and therefore the following limit exists for each $k\ge 1$:
\begin{equation}\label{eq:xik}
\xi_k:=\lim_{n\to\infty}x_{n,k},
\end{equation} 
and clearly $0\le \xi_1\le \xi_2\le\ldots$. These numbers appear on the graph in Section 3.

In \cite[Lemma 2]{Ch1} Chihara proved  the existence of an N-extremal measure
 $\psi_*$, which  is concentrated on the sequence $(\xi_k)$. For the benefit of the reader we give a complete proof in  Theorem~\ref{thm:Fr1}. Chihara also proved that $\psi_*$ has the extremality property, that any other solution $\mu\neq \psi_*$ of the indeterminate Hamburger moment problem satisfies
\begin{equation}\label{eq:Ne6}
\xi(\mu)<\xi(\psi_*)=\xi_1.
\end{equation} 
Chihara does not mention the Friedrichs extension, but \eqref{eq:Ne6} is essential in Pedersen's proof that $\psi_*$ is the Friedrichs extension, because
it has the greatest lower bound among the self-adjoint extensions of the Jacobi matrix.

Chihara's proofs are largely based on chain sequences. The result is given in  Theorem~\ref{thm:Fr2} with a proof independent of chain sequences.
The measure $\psi_*$ equals $\nu_{\alpha(s)}$, so $\xi_k=x_k(\alpha(s)), k=1,2,\ldots$ with the notation of \eqref{eq:Ne1}.

In Section 4 we consider the moment problem corresponding to the shifted Jacobi matrix
$J^\dagger$ formed by the shifted Jacobi parameters $(a_{k+1}), (b_{k+1})$. This means that $J^\dagger$ is obtained from $J$ by deleting the first row and column of $J$. The corresponding normalized Stieltjes moment sequence $s^\dagger$ is  always indeterminate as Stieltjes moment sequence because of \eqref{eq:dgalpha}. This has also been obtained in \cite[Theorem 4.3]{B0} by another  proof.  In Theorem~\ref{thm:final} we finally prove the interlacing of the sequences $(\xi_k)$ and $(\xi^\dagger_k)$ associated with $s$ and $s^\dagger$ via \eqref{eq:xik}.

\section{Preliminaries about indeterminate Hamburger moment problems}

The starting point is a normalized  indeterminate Hamburger moment sequence $s=(s_n)_{n\ge 0}$ as in Section 2.1 of \cite{B:V}, i.e., 
\begin{equation}\label{eq:mom}
s_n=\int x^n\,d\mu(x),\quad \mu\in V, n=0,1,\ldots, 
\end{equation}
where $V$ is the infinite convex set of probability measures on $\R$ with moments of any order and having the moments $s_n$.  As in \cite{B:V} we denote by $P_n(x)$ and $Q_n(x)$ the orthonormal polynomials and the corresponding polynomials of the second kind.
Both polynomial sequences satisfy the difference equation 
\begin{equation}\label{eq:3term}
xy_n=b_ny_{n+1}+a_ny_n+b_{n-1}y_{n-1},\quad n\ge 1
\end{equation}
together with the initial conditions
\begin{equation}\label{eq:initial}
P_0(x)=1, P_1(x)=\frac{1}{b_0}(x-a_0),\quad Q_0(x)=0, Q_1(x)=\frac{1}{b_0}.
\end{equation}
Here
\begin{equation}\label{eq:ab}
a_n=\int xP_n^2(x)\,d\mu(x),\quad b_n=\int xP_n(x)P_{n+1}(x)\,d\mu(x),\quad \mu\in V.
\end{equation}
The infinite matrix
\begin{equation}\label{eq:Jac}
J=\begin{pmatrix}
a_0 & b_0 & 0 & \hdots\\
b_0 & a_1 & b_1 & \hdots\\
0 & b_1 & a_2 & \hdots\\
\vdots &\vdots & \vdots & \ddots
\end{pmatrix},
\end{equation}
is called the Jacobi matrix of the moment problem. The matrix $J$ defines an operator in
$\ell^2$  with the  domain $\mathcal F$ of sequences with only finitely many non-zero elements. The closure of $J$ is a symmetric operator $(T,D(T))$ in $\ell^2$ with deficiency indices $(1,1)$ called the Jacobi operator of $s$. For results about the domain $D(T)$ of the Jacobi operator, and the domains $D(T_t)$ of its self-adjoint extensions see \cite{B:S}. 

 We combine the polynomials $P_n,Q_n$ to form polynomials of two variables, and 
the following formulas hold:
\begin{prop}\cite[Proposition 5.24]{Sch}\label{thm:An-Dn} For $u,v\in\C$ and $n\ge 0$ we have
\begin{eqnarray*}
A_n(u,v)&:=&(u-v)\sum_{k=0}^nQ_k(u)Q_k(v)=
b_n\left|\begin{array}{cc}
 Q_{n+1}(u)&\;Q_{n+1}(v)\\Q_{n}(u)&\;Q_{n}(v)\end{array}\right|\\
B_n(u,v)&:=&-1+(u-v)\sum_{k=0}^n P_k(u)Q_k(v)=
b_n\left|\begin{array}{cc}
P_{n+1}(u)&\;Q_{n+1}(v)\\P_{n}(u)&\;Q_{n}(v)\end{array}\right|\\
C_n(u,v)&:=&1+(u-v)\sum_{k=0}^n Q_k(u)P_k(v)=
b_n\left|\begin{array}{cc}
Q_{n+1}(u)&\;P_{n+1}(v)\\
Q_{n}(u)&\;P_{n}(v)\end{array}\right|\\
D_n(u,v)&:=&(u-v)\sum_{k=0}^nP_k(u)P_k(v)=
b_n\left|\begin{array}{cc}
P_{n+1}(u)&\;P_{n+1}(v)\\
P_{n}(u)&\;P_{n}(v)\end{array}\right|.
\end{eqnarray*}
\end{prop}

It is important to notice that
\begin{equation*}\label{eq:det4}
\left|\begin{array}{cc}\ A_n(u,v)&\;B_n(u,v)\\
 C_n(u,v)&\; D_n(u,v)\end{array}\right|=1\textrm{ for }(u,v)\in\mathbb
C^2,
\end{equation*}
cf. \cite[Equation(5.57)]{Sch}. These polynomials were also introduced in \cite[p. 123]{Ak}.

Since we assume $(s_n)$ to be indeterminate , we know that $(P_n(u)), (Q_n(u))\in\ell^2$ for all $u\in\C$, so we can let $n$ tend to infinity in the expressions in the middle of Proposition~\ref{thm:An-Dn}, and we get four entire functions of two complex variables called the Nevanlinna functions of the indeterminate Hamburger moment problem:

\begin{eqnarray}
A(u,v)&=&(u-v)\sum_{k=0}^\infty Q_k(u)Q_k(v)\label{eq:A}\\
B(u,v)&=&-1+(u-v)\sum_{k=0}^\infty P_k(u)Q_k(v) \label{eq:B}\\
C(u,v)&=&1+(u-v)\sum_{k=0}^\infty Q_k(u)P_k(v)\label{eq:C}\\
D(u,v)&=&(u-v)\sum_{k=0}^\infty P_k(u)P_k(v)\label{eq:D},
\end{eqnarray}
see Section 7.1 in \cite{Sch}. The two-variable  functions were introduced in \cite{Bu:Ca} in a slightly different form, which was subsequently  used in \cite{B}, \cite{Pe}.
If the functions of \cite{Bu:Ca} are marked with a $*$, we have
\begin{eqnarray*}\label{eq:A*-D*}
A^*(u,v)&=&-A(u,v),\; B^*(u,v)=-C(u,v),\\
C^*(u,v)&=&-B(u,v),\; D^*(u,v)=-D(u,v).
\end{eqnarray*}
We clearly have the determinant equation
\begin{equation*}\label{eq:det5}
\left|\begin{array}{cc}\ A(u,v)&\;B(u,v)\\
 C(u,v)&\; D(u,v)\end{array}\right|=1\textrm{ for }(u,v)\in\mathbb
C^2.
\end{equation*}
We  define the following polynomials and entire functions of one variable
\begin{equation}\label{eq:pol-one}
A_n(u)=A_n(u,0),\;B_n(u)=B_n(u,0),\;C_n(u)=C_n(u,0),\;D_n(u)=D_n(u,0),
\end{equation}
\begin{equation}\label{eq:Nev-one}
A(u)=A(u,0),\;B(u)=B(u,0),\;C(u)=C(u,0),\;D(u)=D(u,0).
\end{equation}
The latter are also called the Nevanlinna functions of the moment problem, and they enter in the Nevanlinna parametrization of the indeterminate Hamburger moment problem, $V=\{\nu_\varphi\mid \varphi\in\mathcal N\cup\{\infty\}\}$:
\begin{equation}\label{eq:Nev-par}
\int\frac{d\nu_\varphi(x)}{x-z}=-\frac{A(z)\varphi(z)-C(z)}{B(z)\varphi(z)-D(z)},\quad z\in\C\setminus\R,
\end{equation}
where $\mathcal N$ is the set of Pick functions, i.e., the  holomorphic functions $\varphi:\C\setminus\R\to\C$ satisfying $\Im(\varphi(z))/\Im(z)\ge 0$.  See \cite{Ak} or \cite{B:V} for details. When the Pick function $\varphi\in\mathcal N$ is a constant $t\in\R\cup\{\infty\}$, the solution $\nu_t$ is called N-extremal or a von Neumann solution.
It is well-known that the entire functions $A, B, C, D$ have infinitely many zeros  which are all real and simple.

The two-variable polynomials of Proposition~\ref{thm:An-Dn} can be expressed in terms of the one variable polynomials \eqref{eq:pol-one}   as follows, see \cite[p. 123]{Ak}, or \cite{B:S}:
\begin{prop}\label{thm:2to1} For $u, v\in\C$ we have
 \begin{eqnarray*}
A_n(u,v)&=& \left|\begin{array}{cc}
 A_n(u)&\;A_n(v)\\C_n(u)&\;C_n(v)\end{array}\right|\\
 B_n(u,v)&=& \left|\begin{array}{cc}
 B_n(u)&\;A_n(v)\\D_n(u)&\;C_n(v)\end{array}\right|\\
  C_n(u,v)&=& \left|\begin{array}{cc}
 A_n(u)&\;B_n(v)\\C_n(u)&\;D_n(v)\end{array}\right|\\
  D_n(u,v)&=& \left|\begin{array}{cc}
 B_n(u)&\;B_n(v)\\D_n(u)&\;D_n(v)\end{array}\right|.
 \end{eqnarray*}
\end{prop}
There are similar expressions for the Nevanlinna functions of one and two variables just by letting $n$ tend to infinity.

Differentiating the formula for $A(u,v)$ with respect to $u$  and setting $u=v=x$ we get
\begin{equation}\label{eq:interA}
A'(x)C(x)-C'(x)A(x)=\sum_{k=0}^\infty Q_k^2(x),
\end{equation}
and the same procedure for $D(u,v)$ leads to
\begin{equation}\label{eq:interD}
B'(x)D(x)-D'(x)B(x)=\sum_{k=0}^\infty P_k^2(x).
\end{equation}
For  zeros $\alpha$ of $A$ and $\gamma$ of $C$ we  get from \eqref{eq:interA}
$$
A'(\alpha)C(\alpha)>0,\quad C'(\gamma)A(\gamma)<0,
$$
which shows that $C$ has a zero between any two consecutive zeros of $A$, and that $A$ has a  zero between any two consecutive  zeros of $C$. In other words the zeros  of $A$ and $C$ are interlacing. In the same way \eqref{eq:interD} shows that the zeros of $B$ and $D$ 
are interlacing.

\section{The Stieltjes case}
From now on we assume that $s=(s_n)$ in addition to the previous assumptions is  a Stieltjes moment sequence, so there exists at least one measure $\mu\in V$ supported by the half-line $[0,\infty)$.

There are two possibilities for a Stieltjes moment problem which is indeterminate as Hamburger moment problem. There can be precisely one solution $\mu\in V$ supported by $[0,\infty)$, or there can be more than one such solution and then infinitely many, since they form a convex set. We say that the Stieltjes moment problem is determinate, respectively indeterminate in the sense of Stieltjes in the two cases. In short
the problem is called det(S) or indet(S).

In \cite[Section 4]{B:V} there is a detailed study of the Al-Salam--Carlitz polynomials $V_n^{(a)}(x;q)$, which depends on parameters $a>0$ and $0<q<1$. There is given a complete classification of when the corresponding Hamburger moment problem is determinate or indeterminate, and when the indeterminate Hamburger problem is det(S) or indet(S).  

Since the zeros of $P_n$ and $Q_n$ are located in $(0,\infty)$ and the leading coefficients of $P_n$ and $Q_n$ are positive, we see that for $k\ge 1, x\le 0$
\begin{equation}\label{eq:PPQQ}
P_k(x)P_k(0)>0,\;Q_k(x)Q_k(0)>0,\;P_k(x)Q_k(0)<0,\;Q_k(x)P_k(0)<0,
\end{equation}  
while for $k=0$ the first quantity equals 1 and the three others equal 0.
It follows that $B_n(x)<B_{n+1}(x)<B(x), D_n(x)>D_{n+1}(x)>D(x)$ for $x<0$, and in particular for $x\le 0$:
\begin{equation}\label{eq:zeroBD}
B(x)\ge B_1(x)=-1+xP_1(x)Q_1(0)=\frac{1}{b_0^2}x(x-a_0)-1,\quad D(x)\le D_1(x)=x.
\end{equation}
 From this equation we see that $B$ has a  smallest  zero $\beta_1<0$ and that $D$ has a smallest zero $\delta_1=0$. This means that we can write the zeros 
 $\beta_n,\delta_n, n\ge 1$ of $B$ and $D$ in increasing order, and by the interlacing property mentioned at the end of Section 2 we have
$$
\beta_1<\delta_1=0<\beta_2<\delta_2<\beta_3<\ldots,
$$
and $\beta_n,\delta_n$ tend to infinity.

From \eqref{eq:PPQQ} we similarly get $A_n(x)>A_{n+1}(x)>A(x), C_n(x)<C_{n+1}(x)<C(x)$ for $x<0$, and in particular for $x\le 0$:
\begin{equation}\label{eq:zeroAC}
A(x)\le A_1(x)=\frac{x}{b_0^2},\quad C(x)\ge C_1(x)=1-\frac{a_0}{b_0^2}x\ge 1.
\end{equation}
From this equation we see that $A$ has a  smallest  zero $\alpha_1=0$ and that $C$ has a smallest zero $\gamma_1>0$. This means that we can write the zeros 
 $\alpha_n,\gamma_n, n\ge 1$ of $A$ and $C$ in increasing order, and by the interlacing property we have
$$
\alpha_1=0<\gamma_1<\alpha_2<\gamma_2<\alpha_3<\ldots,
$$ 
 and $\alpha_n,\gamma_n$ tend to infinity.
 
Based on previous work of Chihara, see \cite{Ch1},\cite{Ch2}, the following result was established in
\cite{B0} and \cite[section 2.2]{B:V}:

\begin{prop}\label{thm:BV} Let $s=(s_n)$ be a normalized  indeterminate Hamburger moment sequence which is Stieltjes. The function $j(x)=D(x)/B(x)$ is strictly decreasing in each of the intervals $(-\infty,\beta_1), (\beta_1,\beta_2), (\beta_2,\beta_3),\ldots$ with limits
$$
j(\beta_k-)=-\infty,\quad j(\beta_k+)=\infty,\quad k=1,2,\ldots.
$$
Furthermore, the quantity $\alpha=\alpha(s)$ given by
\begin{equation}\label{eq:alpha}
\alpha=\lim_{x\to -\infty}\frac{D(x)}{B(x)}=\lim_{n\to\infty}\frac{P_n(0)}{Q_n(0)}
\end{equation}
belongs to $(-\infty,0]$, and the sequence $P_n(0)/Q_n(0)$ is strictly increasing.

The Stieltjes problem is det(S) if $\alpha(s)=0$ and indet(S) if $\alpha(s)<0$. 
\end{prop} 

\begin{center}
\includegraphics[scale=0.6]{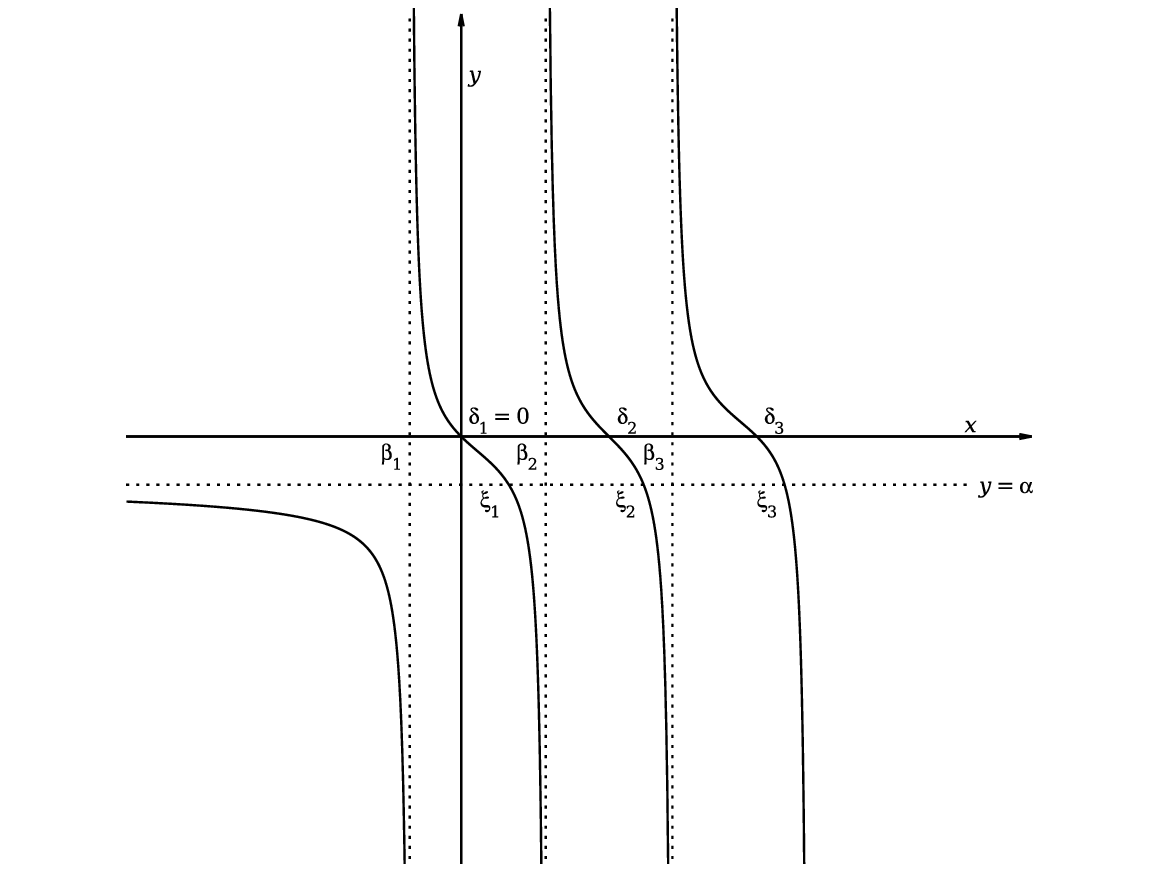}\\
{The graph of $j(x)=D(x)/B(x)$}
\end{center}

\medskip
The quantity $\alpha(s)$ can also be determined from the functions $A, C$:

\begin{prop}\label{thm:AC} Let $s=(s_n)$ be a normalized  indeterminate Hamburger moment sequence which is Stieltjes. The function $l(x)=C(x)/A(x)$ is strictly decreasing in each of the intervals $(-\infty,\alpha_1), (\alpha_1,\alpha_2), (\alpha_2,\alpha_3),\ldots$ with limits
$$
l(\alpha_k-)=-\infty,\quad l(\alpha_k+)=\infty,\quad k=1,2,\ldots.
$$
We have
\begin{equation}\label{eq:alpha*}
\alpha(s)=\lim_{x\to -\infty}\frac{C(x)}{A(x)}.
\end{equation}
\end{prop} 

\begin{proof} We know from Berg--Szwarc \cite{B:S} that $A(z)/C(z)$ is a Pick function so $C(x)/A(x)$ is  strictly decreasing in the intervals $(-\infty,\alpha_1),(\alpha_k,\alpha_{k+1}),k=1,2,\ldots$. (That $A(z)/C(z)$ is a Pick function is also deduced in Section 4 of the present paper).  We have $A(x)<0, B(x)>0$ for $x<\beta_1$, so for those $x$  we have
$D(x)/B(x)-C(x)/A(x)=1/(A(x)B(x))\to 0$ for $x\to -\infty$ because $A(x)\to -\infty$ and
$B(x)\to \infty$ by \eqref{eq:zeroAC} and \eqref{eq:zeroBD}.
This shows \eqref{eq:alpha*}.
\end{proof}

If the Stieltjes problem is det(S), i.e., $\alpha(s)=0$ then $\nu_0$ is the only solution to the Hamburger problem which is concentrated on $[0,\infty[$.  For all other solutions $\mu$ we have $\xi(\mu)<0$.   

Let us now assume that the Stieltjes problem is indet(S), i.e., that $\alpha(s)<0$.

For $t\in\R\cup\{\infty\}$ the N-extremal solutions $\nu_t\in V$ are the discrete  
measures satisfying
\begin{equation}\label{eq:Nev-par-ext}
\int\frac{d\nu_t(x)}{x-z}=-\frac{A(z)t-C(z)}{B(z)t-D(z)},\quad z\in\C\setminus\R,
\end{equation}
so $\nu_t$ is concentrated in the zero-set 
$$
\Lambda_t:=\{x\in\R \mid B(x)t-D(x)=0\}=\{x\in\R \mid j(x)=t\},
$$
which can be described as a strictly increasing sequence $(x_n(t))_{n\ge 1}$.
From Proposition~\ref{thm:BV} we clearly have:

(i) For $t>0$: $x_n(t)\in (\beta_n,\delta_n)$, so $\nu_t$ has one negative mass-point $x_1(t)$  and
the other mass-points are positive. When $t$ increases from 0 to $\infty$, then $x_1(t)$ decreases from $\delta_1=0$ to $\beta_1$ with  $x_1(\infty)=\beta_1$. 

(ii) For $t<\alpha(s)$: $x_1(t)\in (-\infty,\beta_1),\;x_n(t)\in (\delta_{n-1},\beta_n)$ for $n\ge 2$, so $\nu_t$ has one negative mass-point and the other mass-points are positive.
When $t$ increases from $-\infty$ to $\alpha(s)$, then $x_1(t)$ decreases from $\beta_1$ to $-\infty$.

(iii) For $t=\infty$: $x_n(t)=\beta_n$, so $\nu_\infty$ is concentrated in the zero-set of $B$ and has one negative mass-point.

(iv) For $\alpha(s)< t< 0$: $x_n(t)\in (\delta_n,\beta_{n+1})$, so $\nu_t$ is supported by a sequence belonging to $(0,\infty)$. When $t$ increases from $\alpha(s)$ to 0, then $x_1(t)$ decreases from $\xi_1$ to 0.

(v) For $t=0$: $x_n(t)=\delta_n$, so $\nu_0$ is supported by the zero-set of $D$. It is called the Krein solution of the moment problem.

(vi) For $t=\alpha(s)$: $x_n(t)>\delta_n$, and $x_1(\alpha(s))>x_1(t)$ for all $t\neq \alpha(s)$ (including $t=\infty$). The measure $\nu_{\alpha(s)}$ is called the Friedrichs solution of the moment problem. 

That the Friedrichs solution $\nu_{\alpha(s)}$ corresponds to the Friedrichs extension of the  Jacobi matrix $J$  or the Jacobi operator $(T,D(T))$ was established in \cite{Pe1}. In the paper \cite{Pe2} Pedersen gave a very satisfactory characterization of the solutions $\nu_{\varphi}\in V$ which are supported by $[0,\infty)$, in terms of the corresponding Pick function $\varphi$. For the benefit of the reader we repeat it here.

\begin{thm}\cite{Pe2}\label{thm:Ped} The solution $\nu_\varphi\in V$ corresponding to the Pick function $\varphi$ is supported by $[0,\infty)$ if an only if $\varphi$ has a holomorphic extension from $\C\setminus\R$ to $\C\setminus [0,\infty)$ such that 
$$
\alpha(s)\le \varphi(x)\le 0,\quad x<0.
$$
\end{thm}

We shall now describe the Friedrichs solution $\nu_{\alpha(s)}$ in more detail.

\begin{thm}\label{thm:Fr1} The sequence $(\xi_k)$ defined in \eqref{eq:xik} satisfies $0<\xi_1<\xi_2<\ldots<\xi_k\to\infty$ and the $\xi_k$  are the zeros of $B(z)\alpha(s)-D(z)$. The support of the Friedrichs solution is $(\xi_k)$.
\end{thm}

\begin{proof} It is elementary to see that
\begin{equation}\label{eq:el}
-\frac{P_n(z)}{Q_n(0)}\to B(z)\alpha(s)-D(z),\quad z\in\C,
\end{equation}
locally uniformly in $z$, cf. \cite[Equation (8)]{Pe1}, and hence also the derivatives $-P'_n(z)/Q_n(0)$ converge to the derivative $B'(z)\alpha(s)-D'(z)$. Define $y_{n,k}$ such that  $x_{n,k}<y_{n,k}<x_{n,k+1}$ be such that $P_n'(y_{n,k})=0$ for $k=1,\ldots,n-1$.

From \eqref{eq:el} it is clear that each $\xi_k$ is a zero of 
$B(z)\alpha(s)-D(z)$. We claim that $\xi_k<\xi_{k+1}$, for if there was equality we would have $\lim_{n\to\infty}y_{n,k}=\xi_k$, and hence $B'(\xi_k)\alpha(s)-D'(\xi_k)=0$, contradicting that the zeros of $B(z)\alpha(s)-D(z)$ are simple.  In addition $\xi_k\to\infty$ for otherwise the zeros accumulate. 

Finally $B(z)\alpha(s)-D(z)$ has no other zeros than $\xi_k$, for if $\xi$ is a real zero different from all $\xi_k$, then there exists an $\varepsilon>0$ such that
\begin{equation}\label{eq:hlp}
[\xi-\varepsilon,\xi+\varepsilon]\cap [\xi_k-\varepsilon,\xi_k+\varepsilon]=\emptyset, \forall k.
\end{equation}
By the Theorem of Rouch\'e $P_n$ must have zeros in $(\xi-\varepsilon,\xi+\varepsilon)$ for $n$ sufficiently large, but this is impossible because of \eqref{eq:hlp}. 

Clearly $\xi_1\ge 0$, but if $\xi_1=0$ then $D(\xi_1)=0$, hence $B(\xi_1)\alpha(s)=0$, which gives the contradiction $B(\xi_1)=0$ since we assume $\alpha(s)<0$.
\end{proof}

We know that the Friedrich measure $\nu_{\alpha(s)}$ satisfies
\begin{equation}\label{eq:el1}
\xi_1=\xi(\nu_{\alpha(s)}),\quad \alpha(s)=D(\xi_1)/B(\xi_1).
\end{equation}

We have the following result about $\xi_1$, which is the same as \cite[Lemma 2]{Ch1}, except that there is no mentioning of the Friedrichs extension.

\begin{thm}\label{thm:Fr2} Let $s$ be a normalized Stieltjes moment sequence which is indet(S). For any solution $\nu\in V$ to the indeterminate Hamburger problem we have
$\xi(\nu)<\xi(\nu_{\alpha(s)})$ unless $\nu$ is the Friedrichs solution $\nu_{\alpha(s)}$.
\end{thm}

\begin{proof} Suppose that $\nu$ is a solution satisfying $\supp(\nu)\subseteq [\xi_1,\infty)$, and we want to show that $\nu=\nu_{\alpha(s)}$.

Considering the translation $\tau(x)=x-\xi_1$,  the image measures $\tau(\nu_{\alpha(s)}), \tau(\nu)$ are supported on $[0,\infty)$ and belong to the Stieltjes moment problem  with moments
$$
\tilde{s}_n=\int (x-\xi_1)^n\,d\nu(x)=\sum_{j=0}^n\binom{n}{j} (-\xi_1)^js_{n-j}.
$$
It is clearly indeterminate as Hamburger moment problem and the full set of solutions is
$\{\tau(\mu)\mid \mu\in V\}$. The orthonormal polynomials and those of the second kind are $P_n(x+\xi_1),Q_n(x+\xi_1)$. The corresponding Nevanlinna functions marked with  $\tilde{}$ are
$$
\tilde{A}(u,v)=A(u+\xi_1,v+\xi_1),\ldots,\tilde{D}(u,v)=D(u+\xi_1,v+\xi_1),\quad u,v\in\C
$$

From the formulas for the Nevanlinna functions of two variables expressed by the Nevanlinna functions of one variable, see Proposition~\ref{thm:2to1} (where $n\to\infty$), we get
\begin{eqnarray*}
\tilde{A}(z)&=&A(z+\xi_1)C(\xi_1)-C(z+\xi_1)A(\xi_1)\\
\tilde{B}(z)&=&B(z+\xi_1)C(\xi_1)-D(z+\xi_1)A(\xi_1)\\
\tilde{C}(z)&=&A(z+\xi_1)D(\xi_1)-C(z+\xi_1)B(\xi_1)\\
\tilde{D}(z)&=&B(z+\xi_1)D(\xi_1)-D(z+\xi_1)B(\xi_1).
\end{eqnarray*}
These formulas were found in \cite[Proposition 3.3]{Pe1}.
According to these formulas we get using \eqref{eq:el1}
$$
\frac{\tilde{D}(x)}{\tilde{B}(x)}=\frac{D(\xi_1)-B(\xi_1)\frac{D(x+\xi_1)}{B(x+\xi_1)}}
{C(\xi_1)-A(\xi_1)\frac{D(x+\xi_1)}{B(x+\xi_1)}}\to \frac{D(\xi_1)-B(\xi_1)\alpha(s)}
{C(\xi_1)-A(\xi_1)\alpha(s)}=0
$$
for $x\to -\infty$,  and this shows that $\tilde{\alpha}:=\alpha(\tilde{s})=0$ for the new Stieltjes sequence $\tilde{s}$, so $\tilde{s}$ is det(S). We conclude that $\tau(\nu_{\alpha(s)})=\tau(\nu)$
and hence $\nu_{\alpha(s)}=\nu$.
\end{proof}

\begin{rem}{\rm Let $s$ be a normalized indeterminate Hamburger moment sequence and assume that there exists a solution $\mu$ with $\xi(\mu)\in\R$. Then there exists a solution
$\mu_*$ such that $\xi(\mu)<\xi(\mu_*)$ for all solutions $\mu\neq \mu_*$.

This follows by using the translation $\tau(x)=x-\xi(\mu)$ to obtain a Stieltjes moment sequence.
}
\end{rem}

\section{The shifted problem}
If the first row and column of the Jacobi matrix \eqref{eq:Jac} is deleted, we obtain the
matrix

\begin{equation}\label{eq:Jacdag}
J^\dagger=\begin{pmatrix}
a_1 & b_1 & 0 & \hdots\\
b_1 & a_2 & b_2 & \hdots\\
0 & b_2 & a_3 & \hdots\\
\vdots &\vdots & \vdots & \ddots
\end{pmatrix},
\end{equation}
which by Favard's Theorem, cf. \cite[Theorem 5.14]{Sch}, is the Jacobi matrix of a uniquely determined  normalized moment sequence , which we denote $s^\dagger=(s_n^\dagger)$ and call the shifted moment sequence. The shifted problem has been studied by many authors,  and we shall use results from \cite{B0} and \cite{Pe}.

The original moment problem  is Stieltjes if and only if $J$ is positive in the sense that
$\langle Jc,c\rangle\ge 0$ for all sequences $c=(c_n)\in\mathcal F$.
Therefore, if $s$ is Stieltjes, then so is $s^\dagger$.

The orthonormal polynomials  and the associated polynomials for $s^\dagger$ are given as
\begin{equation}\label{eq:dag1}
P^{\dagger}_n(x)=b_0Q_{n+1}(x),\quad Q^\dagger_n(x)=P_1(x)Q_{n+1}(x)-\frac{1}{b_0}P_{n+1}(x),
\end{equation}
and these formulas show that $s$ is indeterminate as a Hamburger problem if and only if $s^\dagger$ is so. They are in \cite{Pe} as well as expressions for  the Nevanlinna functions
 for the shifted problem in terms of those for the original problem. In particular Pedersen found the formulas
 \begin{equation}\label{eq:BDdag}
 B^\dagger(z)=-C(z)-a_0A(z),\quad D^\dagger(z)=b_0^2A(z).
 \end{equation}
 (From this we  get $b_0^2B^\dagger(z)/D^\dagger(z)=-a_0-C(z)/A(z)$, and from this we see that $A(z)/C(z)$ is a Pick function, a property used in Proposition~\ref{thm:AC}.) 
 
 In the following we assume that $s$ is a Stieltjes sequence for an indeterminate Hamburger problem, and we shall relate $\alpha(s)$ and $\alpha(s^\dagger)$. 
 
First of all we know from Proposition~\ref{thm:BV} that the sequence $P_n(0)/Q_n(0)$ is strictly increasing to $\alpha(s)$. In particular
$$
P_1(0)/Q_1(0)=\frac{-a_0/b_0}{1/b_0}=-a_0<\alpha(s),
$$
and hence
\begin{equation}\label{eq:a_0}
\alpha(s)+a_0>0.
\end{equation}

Since $P^\dagger_n(x)=b_0Q_{n+1}(x)$, we know that the zeros $x_{n,k}$ of $P_n$ and 
the zeros $x^\dagger_{n,k}$ of $P^\dagger_n$ satisfy
$$
x_{n,1}<x^\dagger_{n-1,1}<x_{n,2}<x^\dagger_{n-1,2}\ldots<x^\dagger_{n-1,n-1}<x_{n,n},
$$
and therefore with the notation from Theorem~\ref{thm:Fr1}
$$
\xi_1\le \xi^\dagger_1\le \xi_2\le \xi^\dagger_2\le\ldots.
$$
We claim that there are strict inequalities everywhere.

We notice that
\begin{equation}\label{eq:dgalpha}
\alpha(s^\dagger)=-\frac{b_0^2}{\alpha(s)+a_0}<0,
\end{equation}
so the shifted problem is always indet(S), an observation also done in \cite[Remark 4.5]{B0}.

In fact, by \eqref{eq:BDdag} we find
$$
\alpha(s^\dagger)=\lim_{x\to-\infty}\frac{D^\dagger(x)}{B^\dagger(x)}=\lim_{x\to-\infty}\frac{b_0^2A(x)}{-C(x)-a_0A(x)},
$$
and \eqref{eq:dgalpha} follows because $C(x)/A(x)\to \alpha(s)$ for $x\to-\infty$ by
\eqref{eq:alpha*}.

We know that $\xi^\dagger_k$ are the zeros of
$$ 
B^\dagger(z)\alpha(s^\dagger)-D^\dagger(z)=-\frac{\alpha(s) b_0^2}{\alpha(s)+a_0}A(z)+\frac{b_0^2}{\alpha(s)+a_0} C(z).
$$
If $\alpha(s)=0$ this function is proportional to $C(z)$ so $\xi_k$ are the zeros of $D$ and $\xi^\dagger_k$ are the zeros of $C$ and hence $\xi_k\neq \xi^\dagger_k$. If $\alpha(s)<0$
then  $B^\dagger(z)\alpha(s^\dagger)-D^\dagger(z)$ is proportional to $A(z)\alpha(s)-C(z)$, which then has the zeros $\xi^\dagger_k$, while $B(z)\alpha(s)-D(z)$ has the zeros $\xi_k$, and again $\xi_k\neq \xi^\dagger_k$.

Summing up we have proved:

\begin{thm}\label{thm:final} Let $s$ be a normalized indeterminate Hamburger moment sequence which is Stieltjes. The zeros $x_{n,k}$ of $P_n$ and the zeros $x^\dagger_{n,k}$ of $Q_{n+1}$ have limits
$$
\xi_k:=\lim_{n\to\infty}x_{n,k},\quad \xi^\dagger_k:=\lim_{n\to\infty} x^\dagger_{n,k},\quad k=1,2,\ldots
$$  
satisfying
$$
\xi_1 < \xi^\dagger_1 < \xi_2 < \xi^\dagger_2 <\ldots.
$$
The zeros of $B(z)\alpha(s)-D(z)$ are $(\xi_k)$ and the zeros of $A(z)\alpha(s)-C(z)$ are
$(\xi^\dagger_k)$.  
\end{thm}

\begin{rem}{\rm It is a general result about indeterminate Hamburger problems that for each $t\in\R$ the entire functions $A(z)t-C(z)$ and $B(z)t-D(z)$ have infinitely many zeros which are all real and simple and  they interlace.
}
\end{rem}

\noindent
Christian Berg\\
Department of Mathematical Sciences, University of Copenhagen\\
Universitetsparken 5, DK-2100 Copenhagen, Denmark\\
email: {\tt{berg@math.ku.dk}}


\begin{thebibliography}{120}
\bibitem{Ak} N.~I.~Akhiezer, The classical moment problem. Oliver \& Boyd, Edinburgh, 1965.

\bibitem{B0} C.~Berg, Markov's Theorem revisited, J. Approx. Theory {\bf 78} (1994), 260--275.

\bibitem{B} C.~Berg, Indeterminate moment problems  and the theory of entire functions, J. Comput. Appl. Math. {\bf 65} (1995), 27--55.

\bibitem{B:S} C.~Berg, R.~Szwarc, Indeterminate Jacobi operators. To appear in J. Operator Theory. ArXiv:2301.00586.

\bibitem{B:V} C.~Berg, G.~Valent, The Nevanlinna parametrization for some indeterminate Stieltjes moment problems associated with birth and death processes.
Methods and Applications of Analysis {\bf 1} (2) (1994), 169--209.

\bibitem{Bu:Ca} H.~Buchwalter and G.~Cassier, La param\'etrisation de Nevanlinna dans le probl{\`e}me des moments de Hamburger, Expo Math. {\bf 2} (1984), 155--178.

\bibitem{Ch1} T.~S.~Chihara, On indeterminate Hamburger moment problems, Pacific J. Math. {\bf 27} (1968), 475--484.

\bibitem{Ch2} T.~S.~Chihara, Indeterminate symmetric moment problems, J. Math. Anal. appl. {\bf 85} (1982), 331--346.

\bibitem{GKP} G.~K.~Pedersen, {\it Analysis Now}, Graduate Texts in Mathematics Vol. 118. Springer Verlag 1989. 

\bibitem{Pe} H.~L.~Pedersen, The Nevanlinna matrix of entire functions associated with the shifted indeterminate Hamburger moment problem, Math. Scand. {\bf 74} (1994), 152--160.

\bibitem{Pe1} H.~L.~Pedersen, Stieltjes Moment Problems and the Friedrichs Extension of a Positive Definite Operator, J. Approx. Theory {\bf 83} (1995), 289--307.

\bibitem{Pe2} H.~L.~Pedersen, La param\'etrisation de Nevanlinna et le probl{\`e}me des moments de Stieltjes ind\'etermin\'e, Expo Math. {\bf 15} (1997), 273--278.

\bibitem{S:W} A.~D.~Sokal,  J.  Walrad, Continued-fraction characterization of Stieltjes moment sequences with support in $[\xi,\infty)$. ArXiv:2404.12131.
 
\bibitem{Sch} K.~Schm\"{u}dgen, {\it The Moment Problem}, Graduate Texts in Mathematics Vol. 277. Springer International Publishing AG 2017. 
 
\end{thebibliography}
\end{document}